\documentclass{amsart}
\usepackage{fullpage}
\usepackage{amsmath}
\usepackage{amssymb}
\usepackage{latexsym}
\usepackage{amsthm}
\usepackage{graphicx}
\usepackage{multicol}
\newtheorem{theorem}{Theorem}
\newtheorem*{legthrm}{Theorem (Legendre, 1808)}
\newtheorem*{catconj}{Theorem (Mih\u{a}ilescu, 2004, conj. Catalan, 1844)}
\newtheorem*{theorem*}{Theorem}
\newtheorem*{definition*}{Definition}

\newtheorem{proposition}[theorem]{Proposition}
\newtheorem{corollary}[theorem]{Corollary}

\begin{document}
\title{The Greatest Common Divisor of Multinomial Coefficients}
\author{John E. Mosley}

\begin{abstract}

While studying a characteristic number of manifolds we noticed that the calculation was simply computing a multiple of a multinomial coefficient. We were, at the time, interested in computing the greatest common divisor of these characteristic numbers, and from that came this fun and interesting number theory problem. In this short paper, I will give the answer to this problem, which gives the greatest common divisor of certain multinomial coefficients.

\end{abstract}

\thanks{{\bf Acknowledgement:} I would like to thank David Leep for helping me learn to write mathematics clearly, as I hope I have done here, as well as for his help in understanding just how interesting this little problem is. I would also like to thank my advisor, Serge Ochanine and Ben Braun for their helpful comments in preparing this paper.}

\maketitle
\section{Background, Notation, and Basic Definitions}

The $s$-number of a manifold, defined in \cite[\S 16]{Miln}, is a characteristic number that detects whether or not the manifold is indecomposable.  While working on an conjecture similar to Milnor's theorem that every complex cobordism class contains a non-singular algebraic variety \cite[p. 130]{Stong}, we noticed that the computation of the $s$-number of a particular class of SU-manifolds was returning a power of 2 times a certain multinomial coefficient.  Following the proof of Milnor's theorem in \cite{Stong}, we were interested in the greatest common divisor of these characteristic numbers, and so, the greatest common divisor of these particular multinomial coefficients.  Constraints on the class of manifolds we were studying lead to the constraints on the size of parts in the main result, which turned out to be interesting in its own right.\\

So, to begin, let's establish some basic definitions and notation.  Throughout the paper, $p$, $q$, and $r$ will be prime numbers.

\begin{definition*} A {\underline{partition}} of $n\in\mathbb{N}$ is a list of non-negative integers that sum to $n$.  The individual entries in the list are called {\underline{parts}}. \end{definition*}

We will denote the set of all partitions of $n$ by $P(n)$, and the set of all partitions with parts of size at most $n-2$ as $\hat{P}(n)$.  A generic partition contained in $\hat{P}(n)$ will be denoted $\sigma\in \hat{P}(n)$.

\begin{definition*} The {\underline{multinomial coefficient}} $\binom{n}{\sigma_1,\sigma_2,...,\sigma_t}$ is defined by $$(x_1 + x_2+ ... + x_t)^n = \sum_{(\sigma_1, \sigma_2, ..., \sigma_t)\in P(n)} \binom{n}{\sigma_1,\sigma_2,...,\sigma_t}x_1^{\sigma_1}x_2^{\sigma_2}...x_t^{\sigma_t}.$$
\end{definition*}

It will occasionally be convenient to denote the multinomial coefficient $\binom{n}{\sigma_1,\sigma_2,...,\sigma_t}$ associated to the partition $\sigma=(\sigma_1,\sigma_2,...,\sigma_t)$ of $n$ by $\binom{n}{\sigma}$.


\begin{definition*} The {\underline{p-adic expansion of $n$}} is the unique expansion $\displaystyle n=\sum_{i=0}^{\infty} a_ip^i$ with $0 \leq a_i \leq p-1$.\end{definition*}

\begin{definition*} The {\underline{$p$-adic order of $n$}}, denoted $\nu_p(n)$, is the largest power $k$ of $p$ such that $p^k|n$. \end{definition*}

\section{Main Result}

The goal of this paper is to prove the following:

\begin{proposition} $$\displaystyle \gcd_{\sigma\in \hat{P}(n)} \binom{n}{\sigma} = \begin{cases} p &\mbox{if } n=p^s \\ q &\mbox{if } n=q^t+1 \\ p\cdot q &\mbox{if } n=p^s \mbox{ and } n=q^t+1 \\ 1 & \mbox{else} \end{cases} $$
\end{proposition}

Now, recall that $$\binom{n}{\sigma} = \binom{n}{\sigma_1,\sigma_2,...,\sigma_t} = \frac{n!}{\sigma_1!\sigma_2!...\sigma_t!}.$$  So, our goal will be to determine when we can find, for each fixed prime $p < n$, a partition $\sigma \in \hat{P}(n)$ with $$\nu_p(\sigma):=\nu_p(\sigma_1!)+\nu_p(\sigma_2!)+...+\nu_p(\sigma_t!)=\nu_p(n!).$$  First, let's recall from \cite{Leg}, the value of $\nu_p(n!)$.

\begin{legthrm} $\displaystyle \nu_p(n!)=\sum_{i=1}^{\infty}\left\lfloor\frac{n}{p^i}\right\rfloor.$\end{legthrm}

We will also use the following two corollaries of this theorem:

\setcounter{theorem}{0}

\begin{corollary} Let $n=a_0 + a_1p + a_2p^2+...+a_sp^s$ be the $p$-adic expansion of $n$.  Then, $$\nu_p(n!)=a_1\cdot\nu_p(p!)+a_2\cdot\nu_p(p^2!)+...+a_s\cdot\nu_p(p^s!).$$  \end{corollary}

\begin{proof}
We begin with the formula in Legendre's theorem, and replace $n$ by its $p$-adic expansion. $$\nu_p(n!)=\sum_{i=1}^{\infty}\left\lfloor\frac{n}{p^i}\right\rfloor = \sum_{i=1}^{\infty}\left\lfloor\frac{a_0 + a_1p + a_2p^2+...+a_sp^s}{p^i}\right\rfloor.$$
Expanding the sum, we get $$(a_1+a_2p+...+a_sp^{s-1})+(a_2+a_3p+...+a_sp^{s-2})+...+(a_{s-1}+a_sp)+(a_s)$$ $$=a_1+a_2(1+p)+a_3(1+p+p^2)+...+a_s(1+p+...+p^{s-1})$$  $$=\sum_{i=1}^s a_i \sum_{j=0}^{i-1} p^j$$ $$ =\sum_{i=1}^s a_i\sum_{j=1}^{i}\left\lfloor\frac{p^i}{p^j}\right\rfloor$$
$$ = a_1\cdot\nu_p(p!)+a_2\cdot\nu_p(p^2!)+...+a_s\cdot\nu_p(p^s!).$$

\end{proof}

\begin{corollary} $\nu_p(p^m!)=1+p\cdot\nu_p(p^{m-1}!)$ for all $m \geq 1$. \end{corollary}

\begin{proof}
Again, we begin with the formula in Legendre's theorem, apply a few elementary algebraic operations, and mathematical induction.

$$ \nu_p(p^m!) = \sum_{i=1}^{\infty}\left\lfloor\frac{p^m}{p^i}\right\rfloor = \sum_{i=1}^{m}\left\lfloor\frac{p^m}{p^i}\right\rfloor $$
$$= 1 + \sum_{i=1}^{m-1}\left\lfloor\frac{p^m}{p^i}\right\rfloor $$
$$= 1 + p\cdot \sum_{i=1}^{m-1}\left\lfloor\frac{p^{m-1}}{p^i}\right\rfloor $$
$$=1 + p\cdot\nu_p(p^{m-1}!).$$

\end{proof}

\pagebreak

We will also make use of the following proposition:
\setcounter{theorem}{1}
\begin{proposition} Suppose $n=p^s$ or $n=q^t+1$, then $p$ (respectively, $q$) divides $\binom{n}{\sigma}$ for every $\sigma\in\hat{P}(n)$. \end{proposition}

\begin{proof}
    For brevity, we prove only the case that $n=p^s$.  The case that $n=q^t+1$ is proved similarly.\\

    Suppose that $n=p^s$, and let $\sigma\in\hat{P}(n)$. For each part $\sigma_i$ of $\sigma$ we can consider its $p$-adic expansion: $$\sigma_i = a_{i,0}+a_{i,1}p+...+a_{i,(s-1)}p^{s-1}.$$  Note that since $n=p^s$, no $p$-adic expansion of any part of $\sigma\in\hat{P}(n)$ has non-zero coefficient on the $p^s$ term.  Also, since $$n = p^s = \sigma_1 + \sigma_2 + ... + \sigma_f$$ $$=( a_{1,0}+a_{1,1}p+...+a_{1,(s-1)}p^{s-1}) +( a_{2,0}+a_{2,1}p+...+a_{2,(s-1)}p^{s-1}) +...+(a_{f,0}+a_{f,1}p+...+a_{f,(s-1)}p^{s-1})$$ $$=(a_{1,0}+...+a_{f,0})+(a_{1,1}+...+a_{f,1})p+...+(a_{1,(s-1)}+...+a_{f,(s-1)})p^{s-1}$$ ~\\ we can observe that $a_{1,(s-1)}+...+a_{f,(s-1)}$ is at most $p$.  \\

    Now we observe the following two cases.  If $a_{1,(s-1)}+...+a_{f,(s-1)}=p$, then we are in the case presented in Corollary 2, and we have that $$ p\cdot\nu_p(p^{s-1}!) < \nu_p(p^s!).$$ So, $p|\binom{n}{\sigma}$. \\
    
    On the other hand, if $a_{1,(s-1)}+...+a_{f,(s-1)} < p$, then there is some $0<j<s-1$ for which the sum $a_{1,j}+...+a_{f,j} > 1$, and it follows from Corollary 2 that $$\nu_p(\sigma) \leq p\cdot\nu_p(p^{s-2}!)+ (p-1)\cdot\nu_p(p^{s-1}!)< p\cdot\nu_p(p^{s-1}!) < \nu_p(p^s!).$$ So, $p^k|\binom{n}{\sigma}$ for some $k\geq 2$.\\

    Therefore, $p$ divides $\binom{n}{\sigma}$ for every $\sigma\in\hat{P}(n)$.
\end{proof}

Finally, define for each $p$ the $p$-adic partition of $n$ to be $$\sigma_p(n):=\left(\underbrace{p^{s},...,p^{s}}_{a_{s} \mbox{ entries}},\underbrace{p^{s-1},...,p^{s-1}}_{a_{s-1} \mbox{ entries}},...,\underbrace{p,...,p}_{a_{1} \mbox{ entries}},\underbrace{1,...,1}_{a_{0} \mbox{ entries}}\right).$$

\begin{proof}[Proof of Proposition 1]

We are now ready to prove the main result.  The goal, again, is to determine when we can find, for each prime $p<n$, a partition of $n$ whose associated multinomial coefficient is not divisible by $p$.\\

First, suppose that $n$ is neither a prime power nor one more than a prime power. It follows from Corollary 1 that for each prime $p<n$, $p \nmid \binom{n}{\sigma_p(n)}$.  Since, for each prime $p<n$, we have a multinomial coefficient not divisible by $p$, the greatest common divisor of multinomial coefficients over all partitions in $\hat{P}(n)$ is 1.\\

On the other hand, suppose that $n=p^s$ or $n=q^t+1$. Then, we have that $\sigma_p=(p^s)$ or $\sigma_q=(q^t, 1)$.  Note that these partitions are not in $\hat{P}(n)$.  Let's define instead $$\hat{\sigma}_p(n):=\left(\underbrace{p^{s-1},...,p^{s-1}}_{p \mbox{ entries}}\right),$$ and $$\hat{\sigma}_q(n):=\left(\underbrace{q^{t-1},...,q^{t-1}}_{q \mbox{ entries}}, 1\right).$$ 

It follows from Corollary 2 that $p|\binom{n}{\hat{\sigma}_p(n)}$, but $p^2 \nmid \binom{n}{\hat{\sigma}_p(n)}$, and from Proposition 2 that $p|\binom{n}{\sigma}$ for every $\sigma\in\hat{P}(n)$. Similarly, $q|\binom{n}{\hat{\sigma}_q(n)}$, but $q^2 \nmid \binom{n}{\hat{\sigma}_q(n)}$, and $q|\binom{n}{\sigma}$ for every $\sigma\in\hat{P}(n)$. \pagebreak So, if $n=p^s$ or $n=q^t+1$ we can consider the multinomial coefficient associated to the $r$-adic partition for any prime, $r$, less than $n-1$, and $\hat{\sigma}_p(n)$ (respectively, $\hat{\sigma}_q(n)$). Then, the greatest common divisor of multinomial coefficients over all partitions in $\hat{P}(n)$ is $p$ (respectively, $q$).\\

Finally, if $n=p^s$ and $n=q^t+1$, we can consider the multinomial coefficient associated to the $r$-adic partition for any prime, $r$, less than $n-1$, $\hat{\sigma}_p$, and $\hat{\sigma}_q$.  This gives that the greatest common divisor of multinomial coefficients over all partitions in $\hat{P}(n)$ is $p\cdot q$.

\end{proof}

\section{Further Refinement and Interesting Connections}

The penultimate case stated in Proposition 1, when $n=p^s=q^t+1$, is of particular interest.  Of course, if $p^s=q^t+1$, then $p^s-q^t=1$. Solutions to this equation are the subject of Eug\`{e}ne Catalan's famous conjecture from 1844 that was proved by Preda Mih\u{a}ilescu in 2004 \cite{Mih}:

\begin{catconj}
For $p, q$ prime, and $s,t > 1$, the Diophantine equation $p^s-q^t=1$ admits only one solution.  In particular, $3^2-2^3=9-8=1$.

\end{catconj}

We note, however, that there are other solutions when either $s$ or $t$ is 1.  In particular, if $s=t=1$, we have that $$3^1-2^1=1.$$ 

If $s=1$ with $t>1$, we have that $p^1=q^t+1$ must be odd, since if $p=2$, $q^t \leq 1$.  So, $p$ must be an odd prime, and $q=2$.  Primes of this form, $p = 2^t+1,$ are called {\emph{Fermat primes}}.  Some examples of Fermat primes are $$5 = 2^2+1,$$  $$17 = 2^4 + 1,$$  $$257 = 2^8+1,$$ and $$65537=2^{16}+1.$$  In fact, this list, together with the case $3=2^1 + 1$ is the complete list of known Fermat primes. \\

Conversely, if $t=1$ with $s>1$, we have that $p^s=q^1+1$ must be even, since if $q=2$, $p^s=3$, so $p=3$ and $s=1$.  So, $q$ must be an odd prime, and $p=2$. Primes of this form, $q=2^s - 1,$ are called {\emph{Mersenne primes}}. Some examples of Mersenne primes are $$7 = 2^3-1,$$ $$31 = 2^5 - 1,$$ and $$127 = 2^7 - 1.$$  There are currently 48 known Mersenne primes, the largest of which is $2^{57885161}-1$. \\

 Now we can observe the following refinement of Proposition 1.  It is interesting to note that it is currently unknown if there are infinitely many examples satisfying cases 3 and 4 of this corollary.  

\begin{corollary}
$$\displaystyle \gcd_{\sigma\in \hat{P}(n)} \binom{n}{\sigma} = \begin{cases} p &\mbox{if } n=p^s \\ q &\mbox{if } n=q^t+1 \\ 2\cdot q &\mbox{if } n=p^s=q^t+1 \mbox{ and } n \mbox{ even}\\  2\cdot p &\mbox{if } n=p^s=q^t+1 \mbox{ and } n \mbox{ odd}\\ 1 & \mbox{else} \end{cases} $$

\end{corollary}

\end{document}